\documentclass{amsart}
\usepackage[all]{xy}
\usepackage{graphicx,graphics,amssymb,amsmath, amscd, hyperref}
\usepackage[utf8]{inputenc}
\usepackage[abs]{overpic}
\usepackage{verbatim}
\newtheorem{dummy}{anything}[section]
\newtheorem{theorem}[dummy]{Theorem}

\newtheorem{lemma}[dummy]{Lemma}

\newtheorem{question}[dummy]{Question}

\newtheorem{proposition}[dummy]{Proposition}

\newtheorem{corollary}[dummy]{Corollary}
\newtheorem{conjecture}[dummy]{Conjecture}

\theoremstyle{definition}

\newtheorem{remark}[dummy]{Remark}

\newtheorem*{acknowledgements}{Acknowledgements}

\begin{document}

\title{Smoothly non-isotopic Lagrangian disk fillings of Legendrian knots}

\author{Youlin Li and Motoo Tange}

\address{School of Mathematical Sciences, Shanghai Jiao Tong University, Shanghai 200240, China}
\email{liyoulin@sjtu.edu.cn}

\address{Institute of Mathematics, University of Tsukuba, 1-1-1 Tennodai, Tsukuba, Ibaraki 305-8571, Japan
}
\email{tange@math.tsukuba.ac.jp}

\maketitle
\begin{abstract}
In this paper, we construct the first families of distinct Lagrangian ribbon disks in the standard symplectic 4-ball which have the same boundary Legendrian knots, and are not smoothly isotopic or have non-homeomorphic exteriors.
\end{abstract}

\section{Introduction}

Given a ribbon knot $K$ in $S^3$, it may bound distinct ribbon disks. For instance, in \cite{A}, Akbulut constructed a pair of ribbon disks which have the same boundary knots and diffeomorphic exteriors, and are not smoothly isotopic relative to the boundary.

The standard contact 3-sphere $(S^{3}, \xi_{st})$ has a canonical symplectic filling $(B^4,\omega_{st})$, the standard symplectic 4-ball. Given a Legendrian knot $L$ in $(S^{3}, \xi_{st})$, if there is a Lagrangian surface $S$ in $(B^4,\omega_{st})$ whose boundary is $L$, then we say that $S$ is a {\it Lagrangian filling} of $L$ (\cite{CGHS}, \cite{CNS}). There are some restrictions for a Legendrian knot admitting a Lagrangian filling. For instance, the sum of the Thurston-Bennequin invariant of a Lagrangian fillable Legendrian knot and the Euler characteristic of its bounding orientable Lagrangian surface is zero.  See for example \cite{Ch}, \cite{D} and \cite{E0}. The Legendrian unknot with Thurston-Bennequin invariant $-1$ bounds a Lagrangian disk. The Lagrangian disk fillable Legendrian nontrivial knot in $(S^{3}, \xi_{st})$ with minimal crossing number is a Legendrian $\overline{9_{46}}$, the mirror of the knot $9_{46}$. This example appeared in \cite{Ch2}, \cite{HS} and some other papers. There are more Lagrangian disk fillable Legendrian knots listed in \cite[Table 1]{CNS}.

If a Legendrian knot $L$ is Lagrangian fillable, then one can naturally ask how many distinct Lagrangian surfaces filling $L$. 
Using symplectic invariants, people have found many examples of Legendrian knots each of which bounds distinct exact Lagrangian surfaces up to Lagrangian isotopy or Hamiltonian isotopy (\cite{EHK}, \cite{STWZ}, \cite{E}, etc.). In particular, in \cite[Section 3]{E}, Ekholm proved that there are two Lagrangian disks in $(B^4,\omega_{st})$ which are filled by a Legendrian $\overline{9_{46}}$ and are not Hamiltonian isotopic to each other. Ekholm used this distinction to construct a non-loose Legendrian sphere which turns out to be wrong \cite{E1}. However, the erratum in \cite{E1} does not affect the distinction of the two Lagrangian disks. On the other hand, Auroux exhibited a Legendrian knot in the boundary of a Stein domain other than $(B^4,\omega_{st})$ which fills two different Lagrangian disks \cite[Corollary 3.4]{Au}. Those two disks are distinguished by the first homology groups of their exteriors.

In this paper, we construct the first families of distinct Lagrangian disks in $(B^4,\omega_{st})$ which fill the same Legendrian knots in $(S^{3}, \xi_{st})$, and are not smoothly isotopic or have non-homeomorphic exteriors.

\begin{theorem}\label{thm1}
There exist Legendrian knots in $(S^{3}, \xi_{st})$ filling two Lagrangian disks which are not smoothly isotopic relative to the boundary, and have diffeomorphic exteriors.
\end{theorem}

Our first example is a Legendrian knot $\overline{9_{46}}$ shown in Figure \ref{946}. It bounds two Lagrangian disks which are not smoothly isotopic relative to the boundary. Their exteriors are diffeomorphic. These two Lagrangian disks are exactly the same as those in \cite{E}. So Theorem \ref{thm1} implies Ekholm's result. Our second example is a Legendrian knot $\overline{9_{46}}\sharp\overline{9_{46}}$  shown in Figure \ref{fig: LL}. It bounds four Lagrangian disks which are pairwisely not smoothly isotopic relative to the boundary. In particular, some pairs of these four Lagrangian disks have non-homeomorphic exteriors in $B^4$.

\begin{theorem}\label{thm2}
There exists a Legendrian knot in $(S^{3}, \xi_{st})$ filling two Lagrangian disks whose exteriors are not homeomorphic.
\end{theorem}

In \cite{CET}, Conway, Etnyre and Tosun proved that the contact $(+1)$-surgery along a Legendrian knot $L$ in $(S^{3}, \xi_{st})$ is strongly symplectically fillable if and only if $L$ bounds a Lagrangian disk $D$ in $(B^4,\omega_{st})$, and the symplectic filling of the contact $(+1)$-surgery along $L$ can be constructed by removing a neighborhood of $D$ from $B^4$. They asked the following question.

\begin{question}[\cite{CET}]\label{question1}
Let $L$ be a Legendrian knot in $(S^3, \xi_{st})$ with two (or more) distinct Lagrangian disk fillings in $(B^4,\omega_{st})$. Does contact $(+1)$-surgery on $L$ have more than one Stein (or symplectic) filling up to symplectomorphism?
\end{question}

By Theorem \ref{thm2} and the result in \cite{CET}, there is a Legendrian knot along which the contact $(+1)$-surgery has two non-homeomorphic, and hence non-symplectomorphic, Stein fillings. This gives an evidence for Question \ref{question1}.

\begin{corollary}
\label{cor}
Let $L$ be the Legendrian knot $\overline{9_{46}}\sharp\overline{9_{46}}$ depicted in Figure \ref{fig: LL}. Then the contact $(+1)$-surgery on $L$ has two Stein fillings up to homeomorphism.
\end{corollary}

\begin{figure}[htb]
\begin{overpic}
[width=0.25\textwidth]
{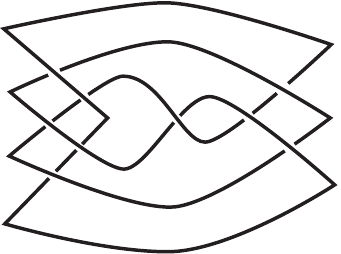}

\end{overpic}
\caption{A Legendrian knot of knot type $\overline{9_{46}}$.}
\label{946}
\end{figure}
\begin{figure}[htb]
\begin{overpic}
[width=0.5\textwidth]
{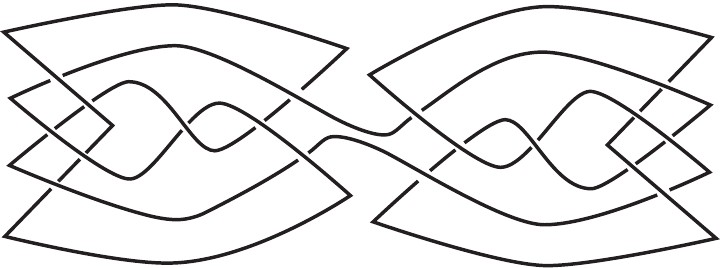}

\end{overpic}
\caption{A Legendrian knot of knot type $\overline{9_{46}}\sharp\overline{9_{46}}$. It is a Legendrian connected sum  of $L$ and itself.}
\label{fig: LL}
\end{figure}

Taking connected sums of multiple copies of Legendrian $\overline{9_{46}}$ in Figure~\ref{946}, we can find arbitrarily many distinct smooth isotopy classes of Lagrangian disks that fill a Legendrian knot.

\begin{theorem}
\label{Nfilling}
For any positive integer $N$ there exists a Legendrian knot $\mathcal{L}$ in $(S^3, \xi_{\text{st}})$ such that the number of smooth isotopy classes of Lagrangian disks that fill $\mathcal{L}$ in $(B^4,\omega_{\text{st}})$ relative to the boundary is greater than $N$.
\end{theorem}

Here we raise a conjecture.
\begin{conjecture}
Let $L$ be a Legendrian ribbon knot in $(S^3, \xi_{\text{st}})$.
Then $L$ bounds only finitely many Lagrangian isotopy types of Lagrangian ribbon disks in $B^4$.
\end{conjecture}
Notice that it is not known whether this conjecture is true or not even in the case of smooth ribbon disks.


\medskip
\begin{acknowledgements}
The authors would like to thank John Etnyre and Honghao Gao for useful conversations. We are also grateful to the referee(s) for valuable suggestions. Part of this work was carried out while the first author was visiting University of Tsukuba and he would like to thank for their hospitality. The first author was partially supported by Grant No. 11871332 of the National Natural Science Foundation of China. The second author was partially supported by JSPS KAKENHI Grant Number 17K14180.
\end{acknowledgements}

\section{Preliminaries}

\subsection {} \label{ribbonmove}
Suppose $K$ is a ribbon knot in $S^3$, and $D$ is a ribbon disk in $B^4$ bounded by $K$. We recall a recipe in \cite[Section 1.4]{Ab} and \cite[Section 6.2]{GS} for drawing a handlebody decomposition of the exterior of $D$ in $B^4$. The ribbon knot can be turned into an unlink through some ribbon moves. A ribbon move means cutting and regluing along a band.  We put dots on each component of the unlink so that they denote 4-dimensional 1-handles. Corresponding to each ribbon move, we add a 4-dimensional 2-handle along an unknot with framing $0$ as shown in the bottom of Figure~\ref{fig: carve}. Then we obtain a handlebody decomposition of the exterior of $D$.

\begin{figure}[htb]
\begin{overpic}
[width=.3\textwidth]
{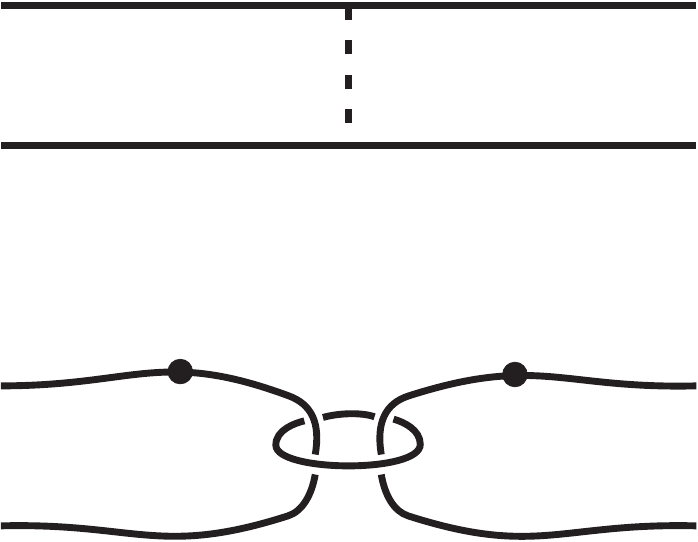}
\put(51,23){$0$}

\end{overpic}
\caption{A local picture of the handlebody decomposition of the exterior of $D$. The dashed line denotes a band along which a ribbon move made. }
\label{fig: carve}
\end{figure}

\subsection {}
A smooth knot $L$ in a contact 3-manifold $(M,\xi)$ is {\it Legendrian} if $L$ is everywhere tangent to $\xi$. A properly embedded smooth surface $S$ in a symplectic 4-manifold $(W, \omega)$ is {\it Lagrangian} if $i^{\ast}\omega=0$, where $i: S\rightarrow W$ is the inclusion map. In this paper, we consider Lagrangian surfaces in $(B^4,\omega_{st})$ which are bounded by a Legendrian knot in $(S^3, \xi_{st})$.   One can construct a Lagrangian surface in $(B^4,\omega_{st})$ whose boundary is a Legendrian link in $(S^3, \xi_{st})$ through a sequence of moves of three types below. See for example \cite{Ch1} and \cite{EHK}.
\begin{itemize}
    \item Adding a maximal Thurston-Bennequin unknot,
    \item Pinch move,
    \item Legendrian isotopy.
\end{itemize}
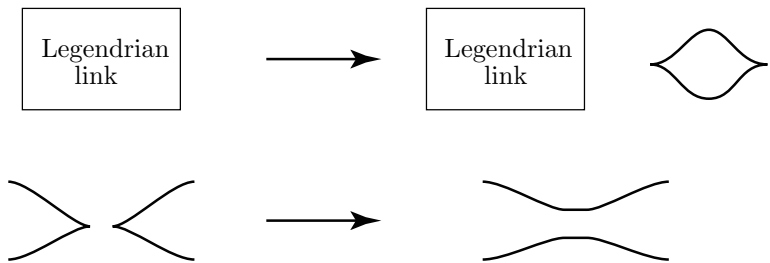
\begin{figure}[htb]
\begin{overpic}
[width=.8\textwidth]
{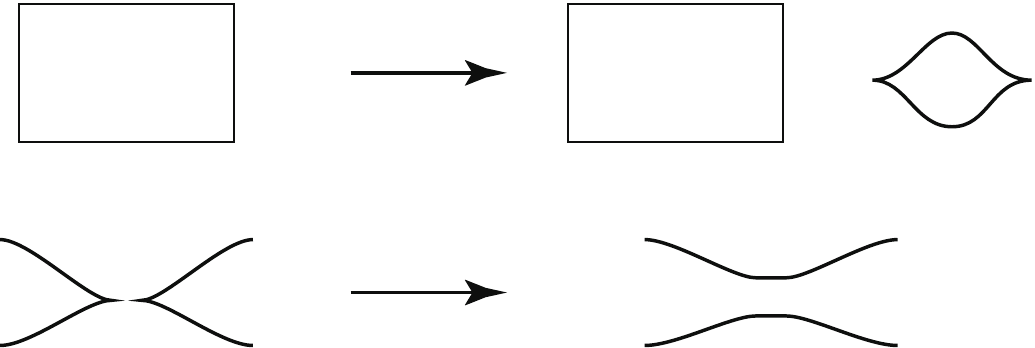}
\put(12.5,77){Legendrian}
\put(25,67){link}
\put(165,77){Legendrian}
\put(180,67){link}
\end{overpic}
\caption{The top arrow stands for adding a Legendrian unknot with maximal Thurston-Bennequin invariant. The bottom arrow is a pinch move.}
\label{fig: moves}
\end{figure}
Hence, if a Legendrian knot can be changed into a Legendrian unlink whose each component has maximal Thurston-Bennequin invariant $-1$ by a sequence of the (inverse) pinch moves and Legendrian isotopies, then the moves give a Lagrangian surface in $(B^4,\omega_{st})$. Such Lagrangian surfaces are said to be {\it decomposable} \cite{CGHS}.

We use the notation $\simeq$ to represent that the two surfaces in $B^4$ are smoothly isotopic relative to boundary.

\subsection {}
Let $L$ be a Legendrian knot in a contact 3-manifold $(Y, \xi)$. Perform $\frac{p}{q}$-surgery with respect to the contact framing, $pq\neq0$, we obtain a closed 3-manifold, and denote it by $Y_{\frac{p}{q}}(L)$. Let $\nu(L)$ be the standard tubular neighborhood of $L$. Extend the contact structure $\xi$ on $Y\setminus \nu(L)$ to $Y_{\frac{p}{q}}(L)$ by a tight contact structure on the new glued-up solid torus, we obtain a contact structure on $Y_{\frac{p}{q}}(L)$, and denote it by $\xi_{\frac{p}{q}}(L)$. This operation is called a \textit{contact $\frac{p}{q}$-surgery}. In \cite{dg}, Ding and Geiges proved that every closed contact 3-manifold can be obtained by contact $(\pm 1)$-surgery along
a Legendrian link in $(S^3, \xi_{st})$. So the contact $(\pm 1)$-surgery plays a fundamental role in the constructions of contact 3-manifolds. The contact $(-1)$-surgery, a.k.a {\it Legendrian surgery}, along a Legendrian knot in $(S^3, \xi_{st})$ yields a Stein fillable contact 3-manifold. However, the contact $(+1)$-surgery does not yield a symplectically fillable contact 3-manifold in general. Conway, Etnyre and Tosun gave a necessary and sufficient condition for a contact $(+1)$-surgery to be strongly symplectically fillable.

\begin{theorem}[\cite{CET}]
Let $L$ be a Legendrian knot in $(S^3, \xi_{st})$. Then contact $(+1)$-surgery on $L$ is strongly symplectically fillable if and only if $L$ bounds a Lagrangian disk in $(B^4,\omega_{st})$. In addition, if $L$ bounds a decomposable Lagrangian disk in $(B^4,\omega_{st})$, then the filling can be taken to be Stein.
\end{theorem}

\subsection {}
We recall a topological characterization of Stein 4-manifolds given by Eliashberg and Gompf.

\begin{theorem}[\cite{G}]
\label{char}
A smooth, compact, connected, oriented 4-manifold $X$ admits a Stein structure (inducing the given orientation) if and only if it can be presented as a handlebody by attaching 2-handles to a framed link in $\partial(D^{4}\cup \text{1-handles})=\sharp m S^{1}\times S^{2}$, where the link is drawn in a Legendrian standard form and the framing coefficient on each link component $K$ is given by $tb(K)-1$.
\end{theorem}

Any Stein 4-manifold has a special property about smooth embedded surfaces by the adjunction inequality coming from Seiberg-Witten theory.
Here, we state the following result.
\begin{theorem}[\cite{LM}]
\label{AMtheorem}
If $X$ is a 4-dimensional Stein manifold and $\Sigma\subset X$ is a closed, connected, embedded surface of genus $g$ in $X$, then
$$[\Sigma]^2+|\langle c_1(X),[\Sigma]\rangle| \le 2g(\Sigma)-2,$$
unless $\Sigma$ is a sphere with $[\Sigma]=0$ in $H_2(X,{\mathbb Z})$.
\end{theorem}
This theorem implies that there is no smoothly embedded square $-1$ sphere in any Stein 4-manifold.

\section{Non-isotopic Lagrangian disks}

The Legendrian knot of knot type $\overline{9_{46}}$ in Figure~\ref{946} has two Lagrangian disk fillings as shown in Figure~\ref{fig: 2lag}. Let $D_1$ and $D_2$ be the two Lagrangian disks, and $W_1$ and $W_2$ the exteriors of $D_1$ and $D_2$ respectively. In \cite{E}, Ekholm shows that $D_1$ and $D_2$ are not Hamiltonian isotopic to each other using a DGA. Our following result implies his result.

\begin{figure}[htb]
\begin{overpic}
[width=\textwidth]
{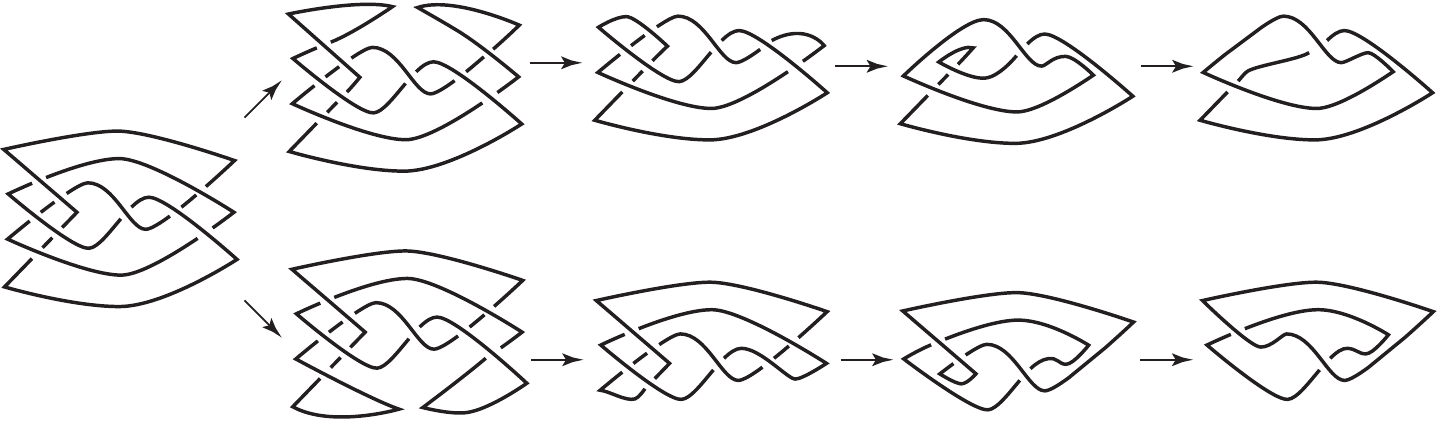}
\end{overpic}
\caption{Two Lagrangian disks $D_1$ and $D_2$. The arrows are either  inverse pinch moves, or Legendrian Reidemeister moves.  }
\label{fig: 2lag}
\end{figure}

\begin{proposition}
\label{D1D2notisotopic}
$D_1\not\simeq D_2$. $W_1$ and $W_2$ are diffeomorphic.
\end{proposition}

\begin{proof}
According to the recipe in Subsection~\ref{ribbonmove},  $W_1$ and $W_2$ have handle decompositions as shown in the top left and the bottom left of Figure~\ref{fig: proof1} respectively. By symmetry, $W_1$ is  diffeomorphism to $W_2$.

By assumption, $\partial D_1$ and $\partial D_2$ are the same knot in $S^3$, so their exteriors in $S^3$ coincide pointwisely. Both $\partial W_1$ and $\partial W_2$ can be seen as the $0$-surgery along $\partial D_1=\partial D_2$.
Extending the identity of the knot exteriors to the $0$-surgeries, we get an identity map $Id:\partial W_1\to \partial W_2$.

Suppose that $D_1\simeq D_2$.
Then the map $Id:\partial W_1\to \partial W_2$ can be extended to a diffeomorphism $f:W_1\to W_2$.
We attach a $(-1)$-framed 2-handle $h$ on $W_1$ as in the top second diagram in Figure~\ref{fig: proof1}. Then $W_1\cup h$ is diffeomorphic to $W_2\cup h$ via $f$. 
Since both $\partial W_1$ and $ \partial W_2$ are identified to the $0$-surgery along $\overline{9_{46}}$, we can replace the handle decomposition of $W_1$ in the top second diagram by that of $W_2$, and obtain the bottom second diagram in Figure~\ref{fig: proof1} which represents
$W_2\cup h$.

\begin{figure}[htb]
\begin{overpic}
[width=\textwidth]
{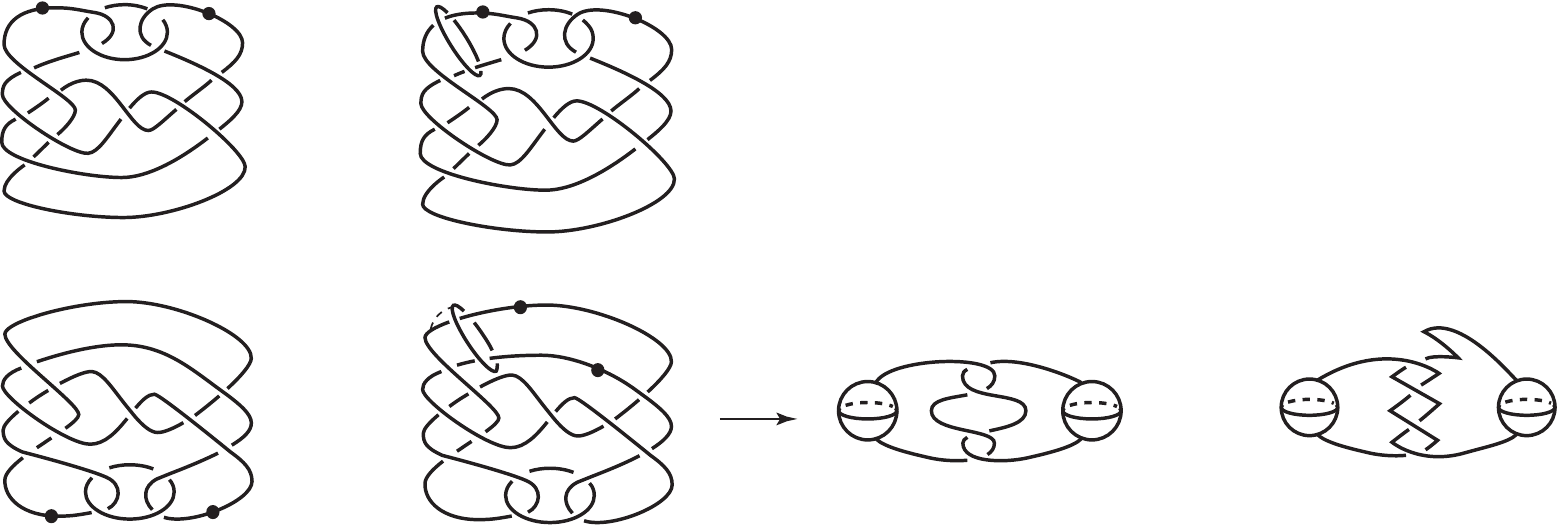}

\put(122, 122){$0$}
\put(96, 122){$-1$}
\put(28, 122){$0$}

\put(100, 52){$-1$}
\put(310, 43){$-1$}
\put(230, 41){$-1$}

\put(28, -7){$0$}
\put(123, -8){$0$}


\end{overpic}
\vspace{3mm}
\caption{The top first diagram is $W_1$. The top second diagram is $W_1\cup h$.  The $(-1)$-framed unknots stand for the attached 2-handles $h$. 
The bottom first diagram is $W_2$. The bottom second diagram is $W_2\cup h$. The bottom horizontal arrow is a cancellation of 1/2-canceling pair. The bottom last diagram is a Stein handlebody. The coefficients $-1$ in the last diagram, and all other Stein handlebody diagrams in this paper, denote Legendrian surgeries. }
\label{fig: proof1}
\end{figure}
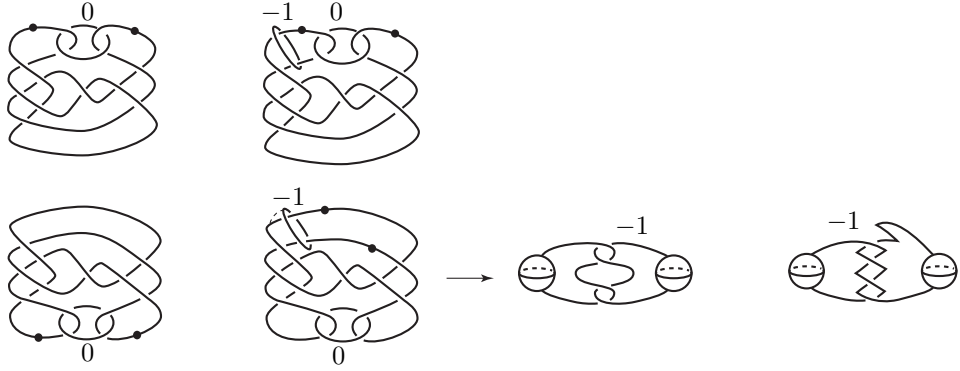


Canceling a 1/2-canceling pair in the handle decomposition of $W_2\cup h$ in the bottom second diagram of Figure~\ref{fig: proof1}, we obtain a simplified handle decomposition of $W_2\cup h$ which is shown in the bottom third diagram of Figure~\ref{fig: proof1}. Furthermore, by Gompf's result in \cite{G}, we can transform it to a Stein handlebody as in the second fourth diagram.
This means $W_2\cup h$ admits a Stein structure.

On the other hand, in $W_1\cup h$, we can find an embedded square $-1$ sphere. Indeed, as shown in Figure~\ref{fig: blowdown}, the handlebody decomposition of $W_1\cup h$ in the top second diagram in Figure~\ref{fig: proof1} can be transformed so that it contains an isolated $(-1)$-framed 2-handle.  
So $W_1\cup h$ never admit any Stein structure by Theorem~\ref{AMtheorem}.

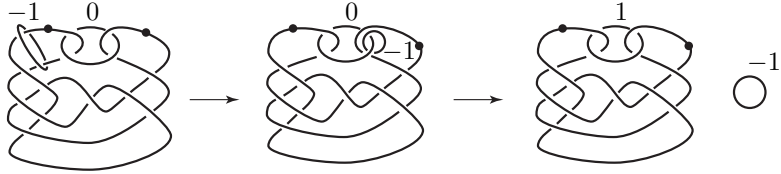
\begin{figure}[htb]
\begin{overpic}
[width=.8\textwidth]
{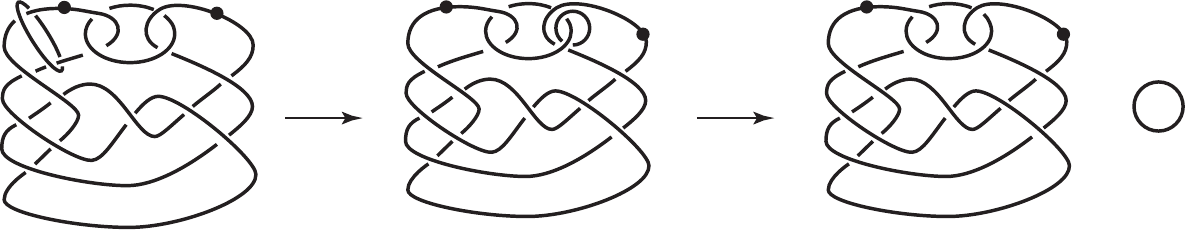}
\put(30,57){$0$}
\put(128,57){$0$}
\put(0,57){$-1$}
\put(142,42.5){$-1$}
\put(230,57){$1$}
\put(280,38){$-1$}
\end{overpic}
\vspace{3mm}
\caption{The first arrow stands for  handle sliding the $(-1)$-framed 2-handle over a 1-handle. The second arrow stands for  handle sliding the $0$-framed 2-handle over the $(-1)$-framed 2-handle. }
\label{fig: blowdown}
\end{figure}

As a result, $W_1\cup h$ and $W_2\cup h$ are not diffeomorphic. This implies $D_1\not\simeq D_2$.
\end{proof}

According to \cite{EHK}, the Legendrian knot of knot type $\overline{9_{46}}\sharp\overline{9_{46}}$ shown in Figure \ref{fig: LL} bounds four Lagrangian disks $D_{11}$, $D_{12}$, $D_{21}$ and $D_{22}$ depicted in Figure \ref{fig: Wij}. Let $W_{ij}$ denote the exteriors of disks $D_{ij}$ in $B^4$, where $i,j=1,2$.

\begin{figure}[htb]
\begin{overpic}
[width=\textwidth]
{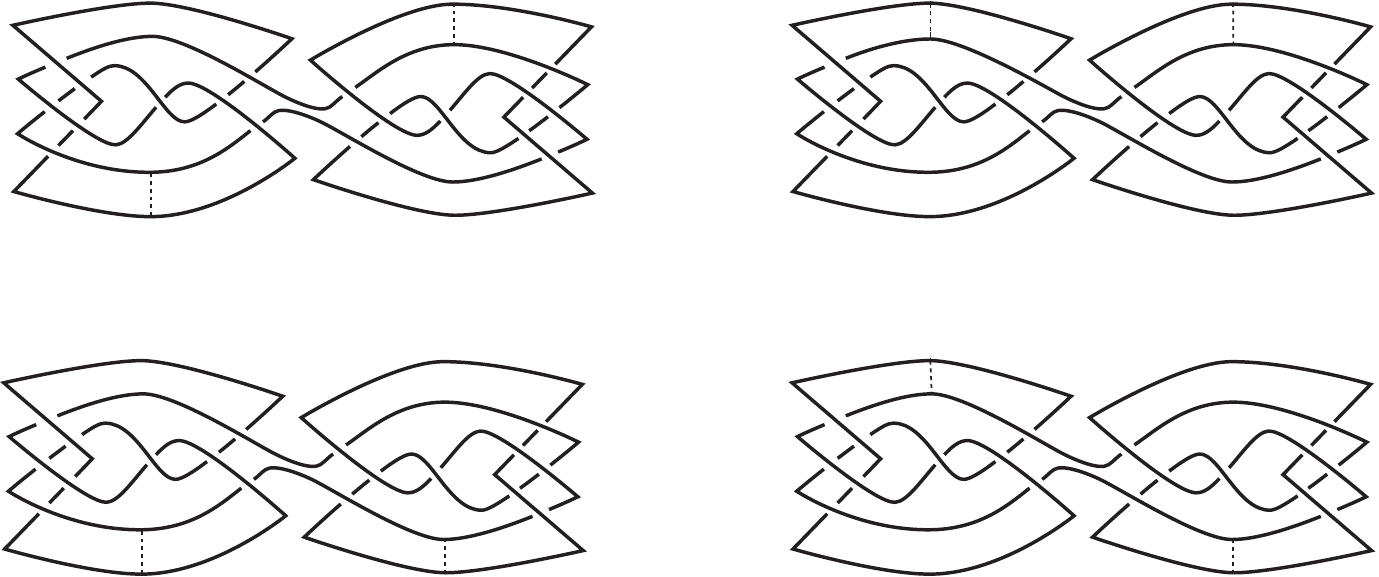}
\put(68, 76){$D_{22}$}
\put(270, 76){$D_{12}$}
\put(68, -8){$D_{21}$}
\put(270, -8){$D_{11}$}
\end{overpic}
\vspace{3mm}

\caption{Four Lagrangian disks $D_{ij}$, where $1\le i,j\le 2$. The dotted lines of each picture present the positions of saddle points of the slice disks.}
\label{fig: Wij}
\end{figure}
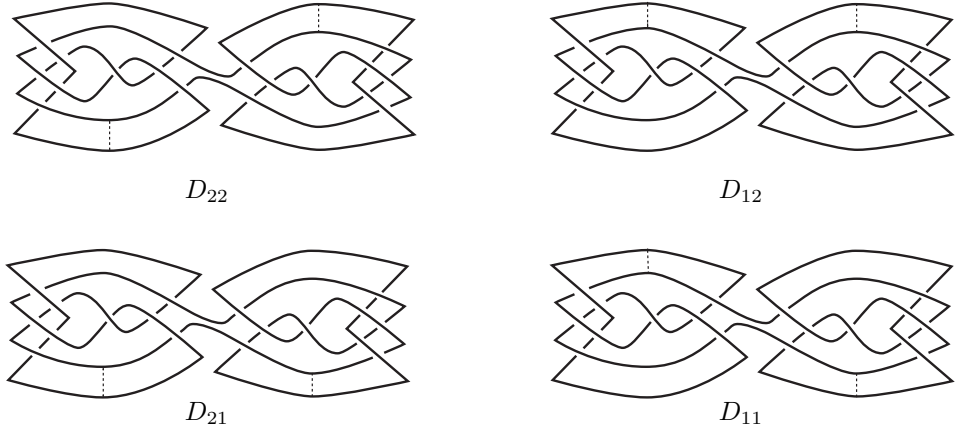

The {\it boundary connected sum} of two properly embedded surfaces $S_1$ and $S_2$ in $B^4$ ($\partial S_i\neq \emptyset$) is removing two disk neighborhoods in $\partial B^4$ of two points in $\partial S_1$ and $\partial S_2$, respectively, and connecting the two arcs of $\partial S_1$ and $\partial S_2$ in the removed disks by an embedded band $I\times I$.
Here $I\times \partial I$ is attached to the two arcs and $\partial I\times I$ is embedded in $\partial B^4$.
In general, if orientations of $\partial S_1$ and $\partial S_2$ are given, then the boundary connected sum of $S_1$ and $S_2$ has two isotopy types. One is the case where the orientation of the boundary of the resulted surface is consistent with the given orientations of $\partial S_1$ and $\partial S_2$, and the other is the inconsistent case.
The result of boundary connected sum of two surfaces $S_1$ and $S_2$ via any band is denoted by $S_1\natural S_2$.
See Figure~\ref{fig: band}.
\begin{figure}[thb]
\begin{overpic}
[width=.9\textwidth]
{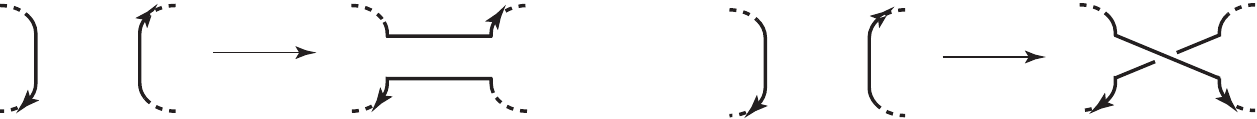}
\put(20, -14){straight band surgery}
\put(-5, 10){$S_1$}
\put(40, 10){$S_2$}
\put(212, -14){twisted band surgery}
\put(182, 10){$S_1$}
\put(227, 10){$S_2$}
\end{overpic}

\vspace{6mm}
\caption{Two types of motion pictures of boundary connected sum of surfaces.
Here, the straight band surgery gives a consistent orientation, while the twisted band surgery gives an inconsistent orientation.}
\label{fig: band}
\end{figure}
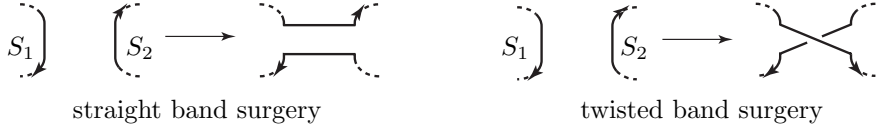

The boundary connected sums of the two Lagrangian disks $D_1$ and $D_2$ have four possibilities depending on the choices of $D_i$ and two ways of boundary connected sums. We remark that $D_{ij}$ in Figure~\ref{fig: Wij} is the boundary connected sum of the Lagrangian disks $D_i$ and $D_j$ for $i,j\in\{1,2\}$.

Here we prove the following:

\begin{proposition}
\label{Dijnotisotopic}
$D_{11}\not\simeq D_{22}$, $D_{12}\not\simeq D_{21}$.  Both of these two pairs of disks have diffeomorphic exteriors.
\end{proposition}

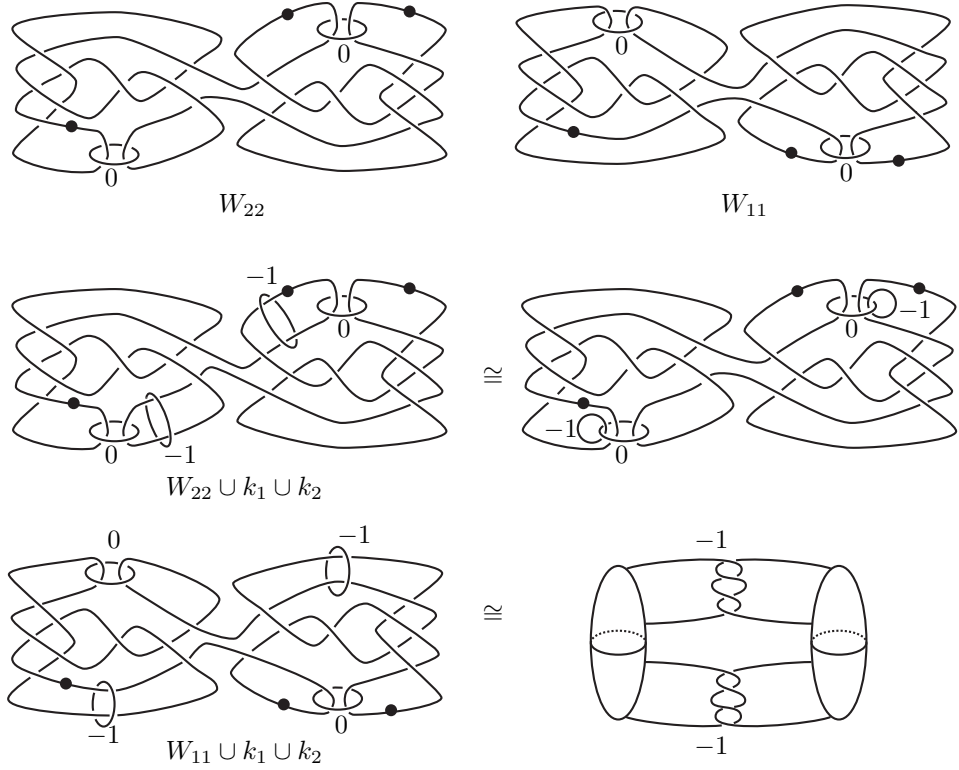
\begin{figure}[htb]
\begin{overpic}
[width=\textwidth]
{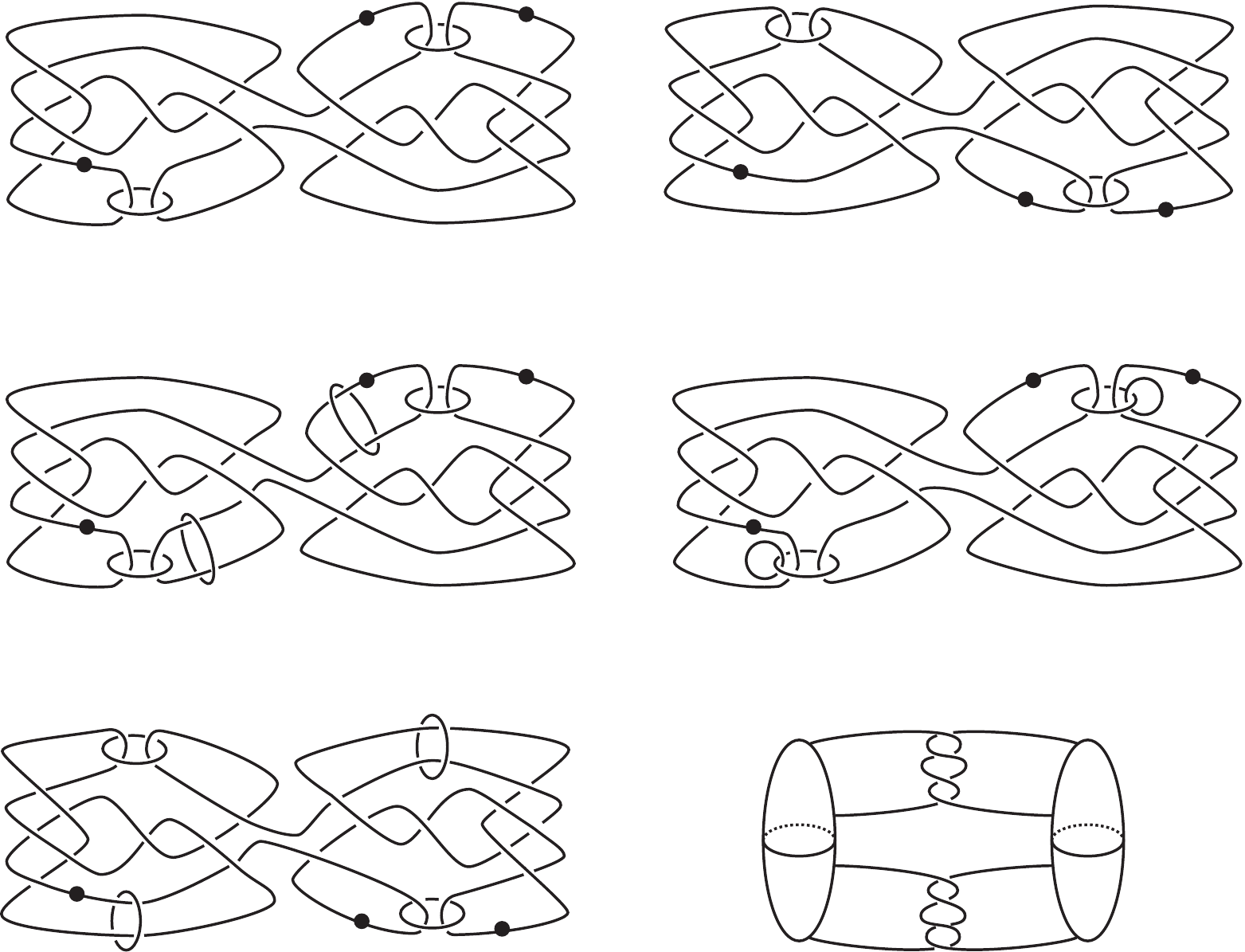}

\put(230, 255){$0$}
\put(125, 252){$0$}
\put(315, 207){$0$}
\put(37, 205){$0$}
\put(80, 195){$W_{22}$}
\put(270, 195){$W_{11}$}

\put(90, 168){$-1$}
\put(337, 155){$-1$}
\put(317, 148){$0$}
\put(125, 148){$0$}
\put(180, 130){$\cong$}
\put(203, 110){$-1$}
\put(230, 100){$0$}
\put(59, 100){$-1$}
\put(37, 100){$0$}

\put(60, 88){$W_{22}\cup k_1\cup k_2$}
\put(125, 70){$-1$}
\put(38, 68){$0$}
\put(260, 68){$-1$}
\put(180, 40){$\cong$}

\put(124, -2){$0$}
\put(30, -6){$-1$}
\put(260, -10){$-1$}
\put(60, -12){$W_{11}\cup k_1\cup k_2$}
\end{overpic}
\vspace{3mm}
\caption{The handle decompositions of $W_{22}$, $W_{11}$, $W_{22}\cup k_1\cup k_2$ and $W_{11}\cup k_1\cup k_2$.}  
\label{fig: D11D22}
\end{figure}

\begin{proof}
According to \cite[Section 1.4]{Ab}, the handle decompositions of $W_{22}$ and $W_{11}$ are shown as in Figure~\ref{fig: D11D22}.
By symmetry of the diagrams, $W_{11}$ is diffeomorphic to $W_{22}$. Similarly, we know that $W_{12}$ is diffeomorphic to $W_{21}$.

We prove that $D_{11}\not\simeq D_{22}$.
The way to prove it is essentially the same as that in the proof of Proposition~\ref{D1D2notisotopic}. We attach two $(-1)$-framed 2-handles $k_1$ and $k_2$ on $W_{22}$.
See the middle first diagram in Figure~\ref{fig: D11D22}.
The two attaching circles can be moved to the position shown in the middle second diagram by a handle sliding.
If $D_{11}\simeq D_{22}$, then by the same argument as in the proof of Proposition~\ref{D1D2notisotopic}, there is a diffeomorphism between the exteriors $W_{11}$ and $W_{22}$ which fixes the boundary.
Replacing the handle decomposition of $W_{22}$ in the middle first diagram by that of $W_{11}$, we get the bottom first diagram. By handle cancellation, we obtain the last diagram which can be deformed into a Stein handlebody in a similar way to that used in the proof of Proposition~\ref{D1D2notisotopic}.
Hence, $W_{11}\cup k_1\cup k_2$ admits a Stein structure.
On the other hand, by Theorem~\ref{AMtheorem} again,  $W_{22}\cup k_1\cup k_2$ never admit any Stein structure. This is because we can find two embedded square $-1$ spheres using the middle second diagram in Figure~\ref{fig: D11D22} and the argument in the proof of Proposition~\ref{D1D2notisotopic}.
This means $D_{11}\not\simeq D_{22}$.

We can also prove $D_{12}\not\simeq D_{21}$ in the similar way.
\end{proof}

In fact, the disks $D_{i1}\not\simeq D_{i2}$ and $D_{1j}\not\simeq D_{2j}$ for $i,j\in \{1,2\}$.  We will prove that the fundamental groups of the exteriors of these pairs are not isomorphic in the next section.


\begin{proof}[Proof of Theorem \ref{thm1}]
It follows from Proposition \ref{D1D2notisotopic} and Proposition \ref{Dijnotisotopic}.
\end{proof}

\section{Non-homeomorphic Lagrangian disk exteriors}

The handle decompositions of $W_{22}$ and $W_{12}$ are depicted  in Figure~\ref{fig: W12}. We prove that  $W_{22}$ and $W_{12}$ have distinct topological types by showing that they have non-isomorphic fundamental groups.

\begin{figure}[htb]
\begin{overpic}
[width=\textwidth]
{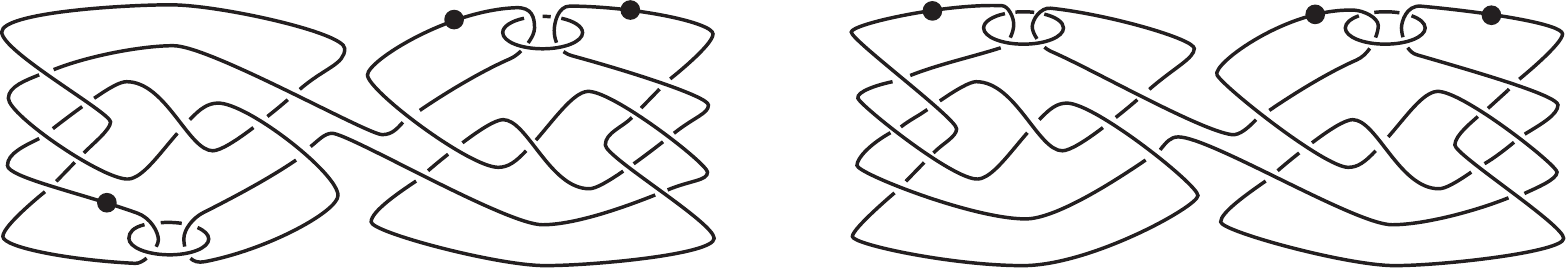}
\put(80, -10){$W_{22}$}
\put(270, -10){$W_{12}$}
\put(37, -7){$0$}
\put(122, 41){$0$}
\put(232, 42){$0$}
\put(316.5, 42){$0$}
\end{overpic}
\vspace{3mm}
\caption{The handle decompositions of $W_{22}$ and $W_{12}$.}
\label{fig: W12}
\end{figure}

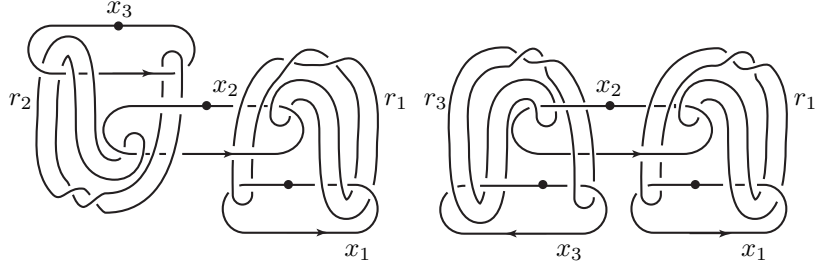
\begin{figure}[htpb]
\begin{overpic}
[width=.8\textwidth]
{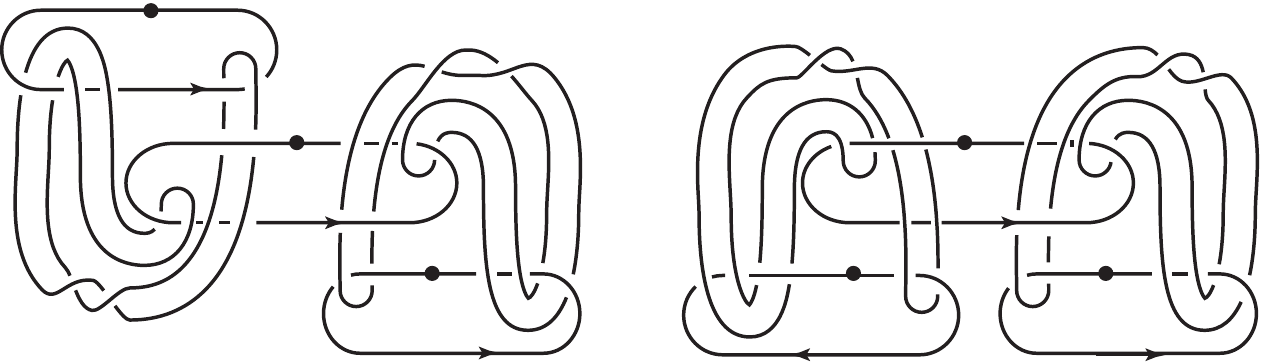}
\put(30, 85){$x_3$}
\put(70, 55){$x_2$}
\put(120, -7){$x_1$}
\put(200, -7){$x_3$}
\put(215    , 55){$x_2$}
\put(270, -7){$x_1$}
\put(135, 50){$r_1$}
\put(-7, 50){$r_2$}
\put(290, 50){$r_1$}
\put(150, 50){$r_3$}
\end{overpic}
\caption{The handle decompositions of $W_{22}$ and $W_{12}$ and the generators and relators of the fundamental groups.}
\label{fig: fund1}
\end{figure}

\begin{lemma}\label{Lem: presentations}
The fundamental groups of $W_{22}$ and $W_{12}$ are computed as follows:
$$\pi_1(W_{22})=\langle x_1,x_2,x_3|x_1x_2x_1^{-1}x_2^{-1}x_1x_2,x_3x_2x_3^{-1}x_2^{-1}x_3x_2\rangle,$$
$$\pi_1(W_{12})=\langle x_1,x_2,x_3|x_1x_2x_1^{-1}x_2^{-1}x_1x_2,x_3x_2^{-1}x_3^{-1}x_2x_3x_2^{-1}\rangle.$$
\end{lemma}

\begin{proof}
Deforming the handle decomposition in Figure~\ref{fig: W12}, we get the diagrams of Figure~\ref{fig: fund1}.
Then we have $$r_1=x_{1}x_{2}x_{1}^{-1}x_{2}^{-1}x_{1}x_{2},$$
$$r_2=x_{3}x_{2}x_{3}^{-1}x_{2}^{-1}x_{3}x_{2},$$
$$r_3=x_{3}x_{2}^{-1}x_{3}^{-1}x_{2}x_{3}x_{2}^{-1}.$$
The presentations of $\pi_1(W_{22})$ and $\pi_1(W_{12})$ are $\langle x_1,x_2,x_3|r_1,r_2\rangle$ and $\langle x_1,x_2,x_3|r_1,r_3\rangle$, respectively.
\end{proof}

\begin{lemma}\label{Lem: nonisomorphic}
$\pi_1(W_{22})$ and $\pi_1(W_{12})$ are not isomorphic.
\end{lemma}

\begin{proof}


Let $t$ be a generator of $\mathbb{Z}$. We consider the abelianizations $\alpha:\pi_1(W_{22})\to \langle t\rangle\cong {\mathbb Z}$ and $\beta:\pi_1(W_{12})\to \langle t\rangle\cong {\mathbb Z}$ as follows,
\[\alpha:\ \left\{
             \begin{array}{lr}
             x_1,x_3\mapsto t, &  \\
             x_2\mapsto t^{-1}, &
             \end{array}
\right.\ \
\text{and}\ \
\beta:\ \left\{
             \begin{array}{lr}
             x_1\mapsto t, &  \\
             x_2,x_3\mapsto t^{-1}. &
             \end{array}
\right.
\]

The abelianizations of the Fox derivatives for $r_1,r_2,r_3$ are
$$\begin{cases}
\alpha(\partial_{x_1} r_1)=\beta(\partial_{x_1} r_1)=2-t^{-1},&  \alpha(\partial_{x_{2}} r_{1})=\beta(\partial_{x_{2}} r_{1})=2t-1,\\\alpha(\partial_{x_2} r_2)=2t-1,& \alpha(\partial_{x_3} r_2)=2-t^{-1}\\
\beta(\partial_{x_2} r_3)=-2+t,& \beta(\partial_{x_3} r_3)=2-t,
\end{cases}$$
and, obviously, $\partial_{x_{3}} r_{1}=\partial_{x_{1}} r_{2}=\partial_{x_{1}} r_{3}=0$.
So the presentation matrix of $\pi_1(W_{22})$ is the following
\[\alpha:\
  \left[ {\begin{array}{ccc}
   \partial_{x_1} r_1 & \partial_{x_2} r_1 & \partial_{x_3} r_1 \\
   \partial_{x_1} r_2 & \partial_{x_2} r_2 & \partial_{x_3} r_2 \\
  \end{array} } \right]\mapsto\left[ {\begin{array}{ccc}
   2-t^{-1} & 2t-1 & 0 \\
   0 & 2t-1 & 2-t^{-1} \\
  \end{array} } \right],
\]
and the Alexander polynomial of $\pi_1(W_{22})$ is $(2-t^{-1})^2$.

In the same way, the presentation matrix of $\pi_1(W_{12})$ is as follows:
\[\beta:\
\left[ {\begin{array}{ccc}
   \partial_{x_1} r_1 & \partial_{x_2} r_1 & \partial_{x_3} r_1 \\
   \partial_{x_1} r_3 & \partial_{x_2} r_3 & \partial_{x_3} r_3 \\
  \end{array} } \right]\mapsto\left[ {\begin{array}{ccc}
   2-t^{-1} & 2t-1 & 0 \\
   0 & -2+t & 2-t \\
  \end{array} } \right],
\]
and the Alexander polynomial of $\pi_1(W_{12})$ is $(2-t^{-1})(2-t)$.

Since $\pi_1(W_{22})$ and $\pi_1(W_{12})$ have different Alexander polynomials, they are not isomorphic.
\end{proof}

\begin{proof}[Proof of Theorem \ref{thm2}]
It follows from Lemma \ref{Lem: presentations} and Lemma \ref{Lem: nonisomorphic}.
\end{proof}

\section{Arbitrarily many non-isotopic Lagrangian disks}

In this section we construct arbitrarily many smoothly non-isotopic Lagrangian fillings for Legendrian knots. At first, we give a Stein handlebody decomposition of $W_2$, the exterior of the disk $D_2$ in $B^4$.

\begin{lemma}\label{Lem: Stein946}
There is a Stein handle decomposition of $W_2$ shown as in the bottom last diagram in Figure~\ref{fig: Stein946}.
\end{lemma}

\begin{proof}
We isotope the handle decomposition of $W_2$ shown as in Figure~\ref{fig: proof1} to the top two diagrams in Figure \ref{fig: Stein946}. Then we change the dotted circle presentation of the 1-handles to the ordinary presentation, and transform the smooth handle decomposition in the bottom second diagram to a Stein handlebody in the last diagram in Figure~\ref{fig: Stein946}. Note that the Legendrian knot in the bottom last diagram has Thurston-Bennequin invariant $1$.
\end{proof}

\begin{figure}[htpb]
\begin{overpic}
[width=.8\textwidth]
{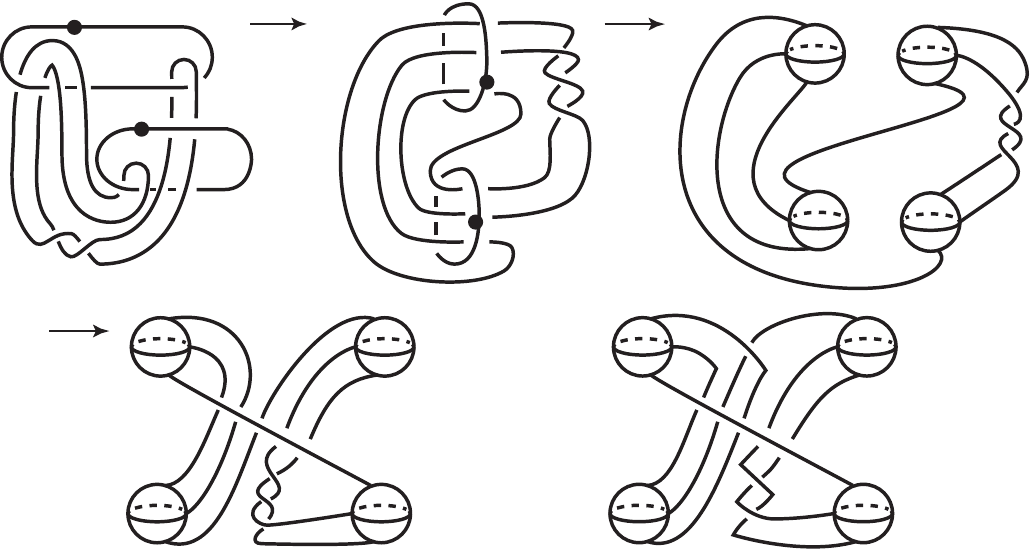}

\put(0, 85){$0$}
\put(95, 140){$0$}
\put(240, 120){$0$}

\put(70, 60){$0$}
\put(202, 63){$-1$}

\end{overpic}
\caption{A Stein structure on $W_2$, the exterior of the Lagrangian disk $D_2$ in $B^4$.  All these arrows are diffeomorphisms.}
\label{fig: Stein946}
\end{figure}

\begin{figure}
\begin{overpic}
[width=.9\textwidth]
{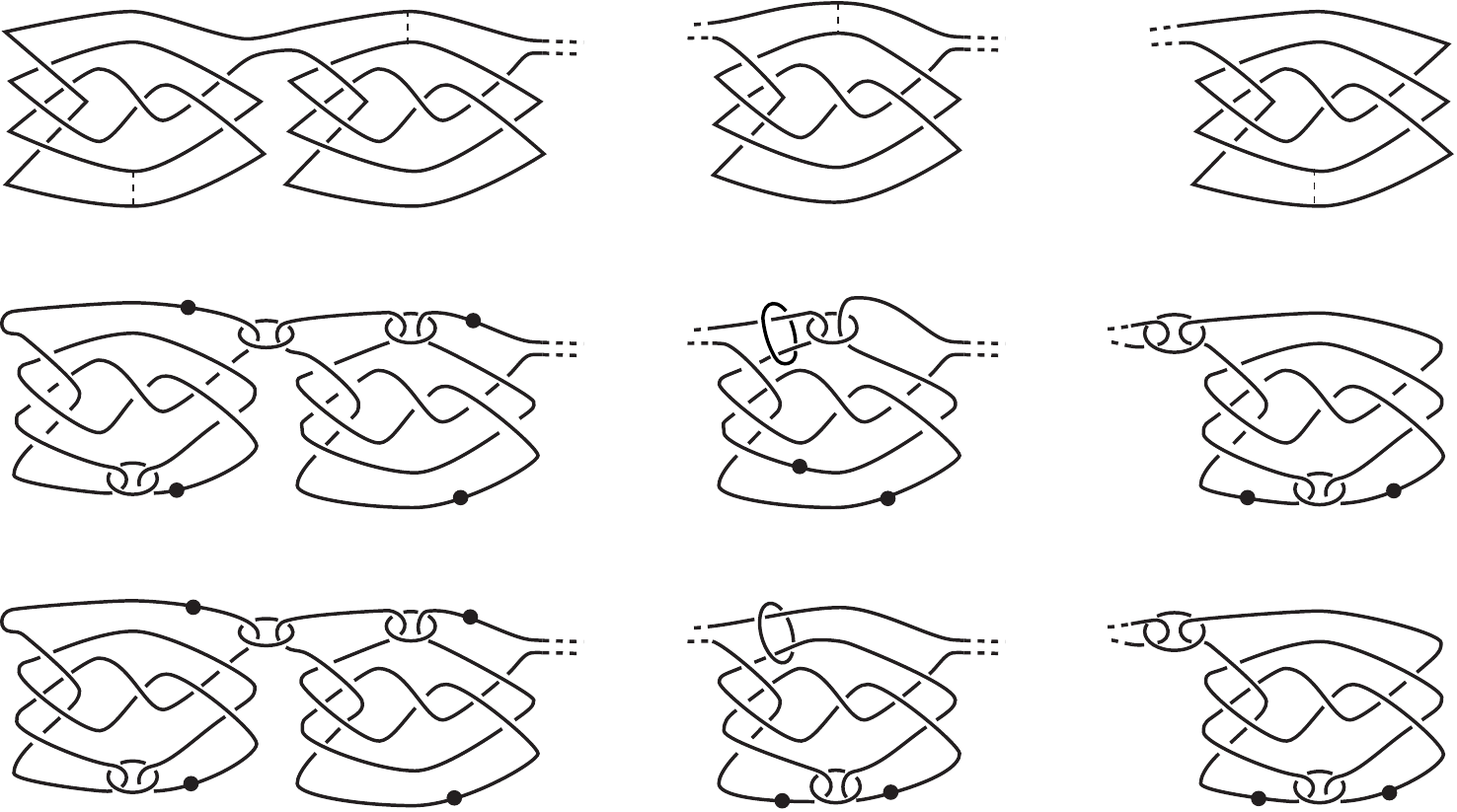}

\put(58, 115){$0$}
\put(28, -5){$0$}
\put(28, 60){$0$}
\put(90, 115){$0$}
\put(163, 116){$-1$}
\put(260, 115){$0$}
\put(58, 47){$0$}
\put(90, 47){$0$}
\put(183, -7){$0$}
\put(163, 50){$-1$}
\put(260, 47){$0$}
\put(290, 60){$0$}
\put(290, -7){$0$}
\put(182, 115){$0$}
    \end{overpic}
    \caption{The top diagram is $D_{i_1i_2\cdots i_n}$. The middle and the bottom diagrams are handle decompositions of $W\cup h$ and $W'\cup h$, respectively.}
    \label{fig: naturalprod}
\end{figure}

\begin{figure}[htb]
\begin{overpic}
[width=.55\textwidth]
{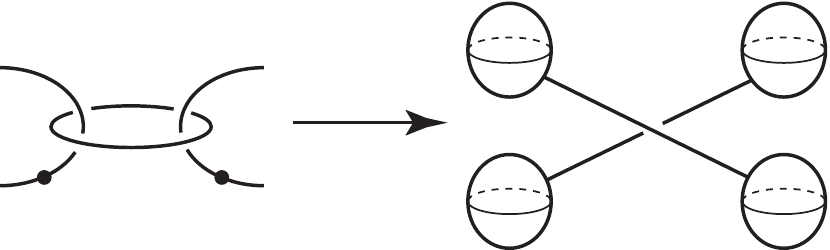}
\put(25, 40){$0$}
\put(140, 40){$-1$}
\end{overpic}
\caption{A 0-framed 2-handle connecting 1-handles and its Stein handlebody.}
\label{fig: 2handle}
\end{figure}

\begin{figure}[htpb]
\begin{overpic}
[width=\textwidth]
{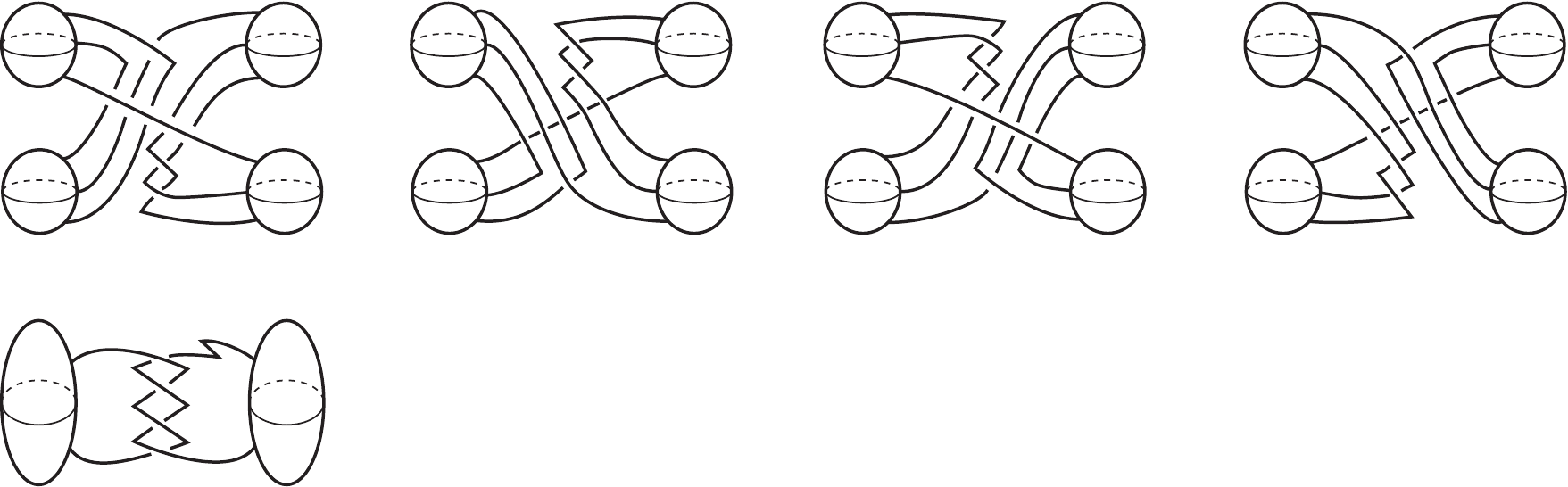}

\put(20, 110){$-1$}
\put(120, 110){$-1$}
\put(215, 110){$-1$}
\put(310, 110){$-1$}
\put(20, 35){$-1$}
\end{overpic}
\caption{The diagrams of the first row are rotations of Stein handlebody diagram in Figure~\ref{fig: Stein946}. The diagram in the second row is a handle decomposition of $W_2$ or $W_1$ union with a 2-handle $h$. }
\label{fig: L}
\end{figure}

\begin{figure}[htpb]
\begin{overpic}
[width=.75\textwidth]
{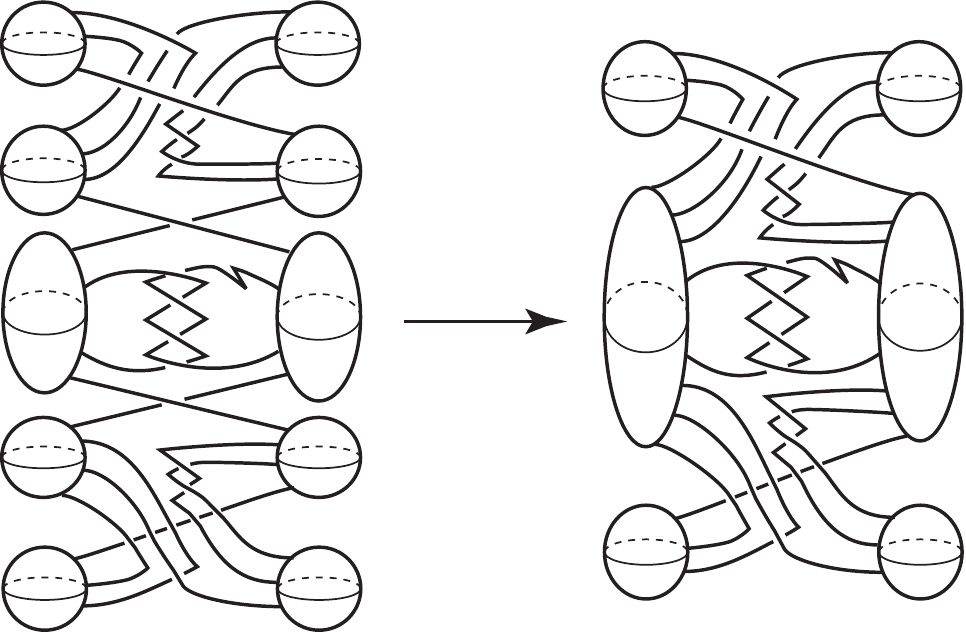}
\put(25, 175){$-1$}
\put(25, 90){$-1$}
\put(20, 112){$-1$}
\put(18, 61.5){$-1$}
\put(32, 2){$-1$}

\put(113,100){cancellation}
\put(200, 10){$-1$}
\put(195,90){$-1$}
\put(200,165){$-1$}

\end{overpic}
\caption{A Stein structure on the complement of the Lagrangian disk $D_{211}$ of $\mathcal{L}_3$ union with a 2-handle $h$.}
\label{fig: Connect}
\end{figure}

Now we give a proof of Theorem \ref{Nfilling}.
\begin{proof}[Proof of Theorem \ref{Nfilling}]

Taking Legendrian connected sum of $n$ copies of the Legendrian $\overline{9_{46}}$ in Figure \ref{946}, we get a Legendrian knot as in the first row of Figure~\ref{fig: naturalprod}. We denote it by $\mathcal{L}_n$.

Choosing one of the two pinch positions per a connected sum summand,
we can get $2^{n}$ Lagrangian disks $D_{i_1i_2\cdots i_n}$, where $i_j\in \{1,2\}$.
In fact, $D_{i_1i_2\cdots i_n}$ is the boundary connected sum $D_{i_1}\natural D_{i_2}\natural\cdots\natural D_{i_n}$ of $D_{i_1}$, $D_{i_2}$, $\cdots$, and $D_{i_n}$.
The first row in Figure~\ref{fig: naturalprod}
represents one example of Lagrangian disks by indicating $n$ positions to pinch.
By Lemma \ref{Lem: Stein946}, the exterior of this disk has a Stein structure which can be drawn by piling vertically
$n$ Stein handlebody diagrams of the four patterns
in the first row of Figure \ref{fig: L} according to the $n$-tuple $(i_1, i_2, \cdots, i_n)$,
and put $n-1$ $0$-framed 2-handle to connect adjacent summands as in Figure~\ref{fig: 2handle}.

Suppose that
$D_{i_1i_2\cdots i_n}\simeq D_{j_1j_2\cdots j_n}$, where the two $n$-tuples $(i_1, i_2, \cdots, i_n)$ and $(j_1, j_2, \cdots, j_n)$ in $\{1,2\}^n$ are distinct. Then there exists $r\in\{1,\cdots, n\}$ such that $i_r\neq j_r$, say $i_r=1$ and $j_r=2$.
Let $W$ and $W'$ denote the exteriors of two disks $D_{i_1i_2\cdots i_n}$ and $D_{j_1j_2\cdots j_n}$ respectively.
On the $r$-th summand we put a $(-1)$-framed 2-handle $h$ (see the second row in Figure~\ref{fig: naturalprod}) in a similar way to that depicted in Figure~\ref{fig: proof1}.
By our assumption, $W\cup h$ and $W'\cup h$ should be diffeomorphic by the same reason as the one in the proof of Proposition~\ref{D1D2notisotopic}.

At first we claim that  $W'\cup h$ admits a Stein structure. We can perform a cancellation of $1/2$-canceling pair as in the proof of Proposition \ref{D1D2notisotopic}. The Stein handlebody diagram of $W'\cup h$ can be constructed by
piling $n$ choices of the five patterns in Figure~\ref{fig: L} vertically, taking the last Stein diagram exactly once, and connecting them by $n-1$ 0-framed 2-handles. See Figure \ref{fig: Connect} for an example of Stein handlebody decomposition of $W'\cup h$.

Then we claim that $W\cup h$ does not admit any Stein structure. We isotope the middle diagram of Figure~\ref{fig: naturalprod} to the top diagram of Figure~\ref{fig: blowdown1}. Then, by handle sliding and blowdown, we can see that $W\cup h$ contains an embedded square $-1$ sphere. By  Theorem~\ref{AMtheorem}, $W\cup h$ does not admit any Stein structure.

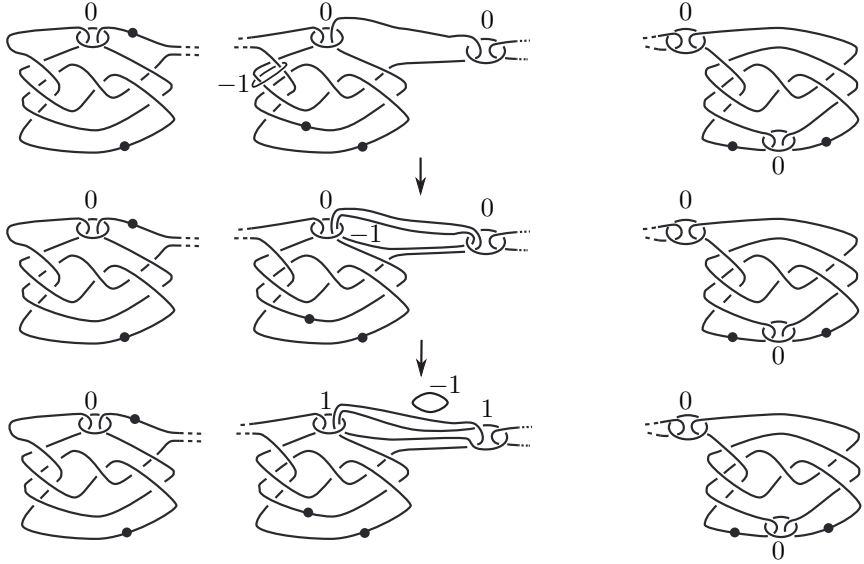
\begin{figure}
\begin{overpic}
[width=0.9\textwidth]
{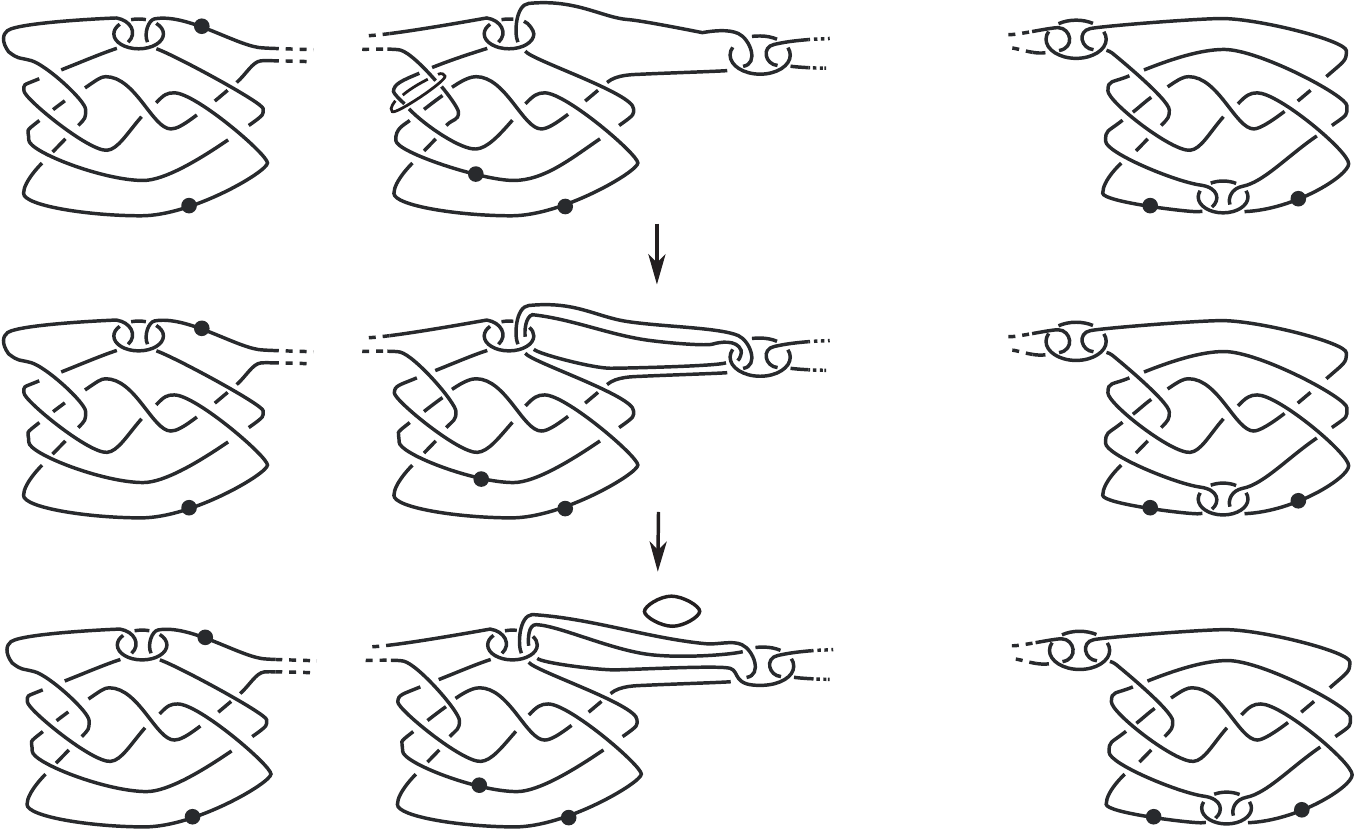}

\put(30, 197){$0$}
\put(119, 197){$0$}
\put(180, 194){$0$}
\put(255, 197){$0$}
\put(80, 170){$-1$}
\put(290, 139){$0$}

\put(30, 126){$0$}
\put(119, 126){$0$}
\put(180, 123){$0$}
\put(255, 126){$0$}
\put(130, 113){$-1$}
\put(290, 67){$0$}

\put(30, 50){$0$}
\put(119, 50){$1$}
\put(180, 47){$1$}
\put(255, 50){$0$}
\put(160, 56){$-1$}
\put(290, -7){$0$}

\end{overpic}
\caption{The first arrow stands for handle sliding the $(-1)$-framed 2-handle over a 1-handle. The second arrow is a blowdown. }
\label{fig: blowdown1}
\end{figure}

Thus, $W\cup h$ and $W'\cup h$ are not diffeomorphic and this means that $D_{i_1i_2\cdots i_n}\not\simeq D_{j_1j_2\cdots j_n}$.

Therefore, for any number $N$, the Legendrian knot $\mathcal{L}_n$ where $n$ is an integer with $n>\log_{2}N$, has at least $N+1$ mutually smoothly non-isotopic Lagrangian ribbon disks in $B^4$.
\end{proof}

\section{Relationship with the Akbulut cork}
Let $C$ be a contractible 4-manifold depicted in Figure~\ref{fig: akbulut}.
There is an involution on $\partial C$ which is the $180^\circ$ rotation about the horizontal axis. The pair of $C$ and the involution is called an {\it Akbulut cork}. This involution can extend to a homeomorphism on $C$, but cannot extend to any diffeomorphism on $C$. There are two slice disks in $B^4$ constructed by the cork twist of $C$ in the following way.
\begin{figure}[htpb]
\begin{overpic}
[width=.2\textwidth]
{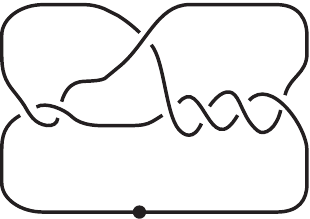}
\put(75, 50){$0$}
\end{overpic}
\caption{Akbulut cork $C$.}
\label{fig: akbulut}
\end{figure}




Akbulut constructed two smoothly non-isotopic slice disks $d_1$ and $d_2$ of a common knot $K_{Ak}$ (Figure 7 in \cite{A}) in $B^4$.  In fact, each exterior of $d_i$ ($i=1,2$) in $B^4$ is diffeomorphic to the exterior of a slice disk in $C$. The property of the cork implies that $d_1\not\simeq d_2$.

In the proof of \cite[Theorem 2]{AY}, Akbulut and Yildiz proved that there is an involution of the boundary of the exterior of ribbon disk $d_1$ (or $d_2$) in $B^{4}$ extends to a homeomorphism, but does not extend to any diffeomorphism of the exterior. Here we consider an involution on $\partial W_1$ and the extendability.

\begin{proposition}
\label{Cork}
Let $\iota$ be an involution on $\partial W_1=S^3_0(\overline{9_{46}})$ as in Figure~\ref{fig: involution}, i.e.,
the $180^\circ$ rotation about horizontal axis.
Then the action cannot extend to any self-homotopy equivalence on $W_1$.
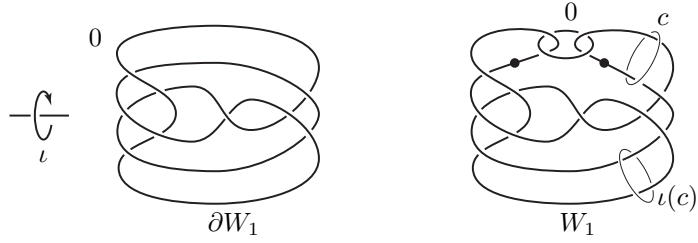
\begin{figure}[htb]
\begin{overpic}
[width=.7\textwidth]
{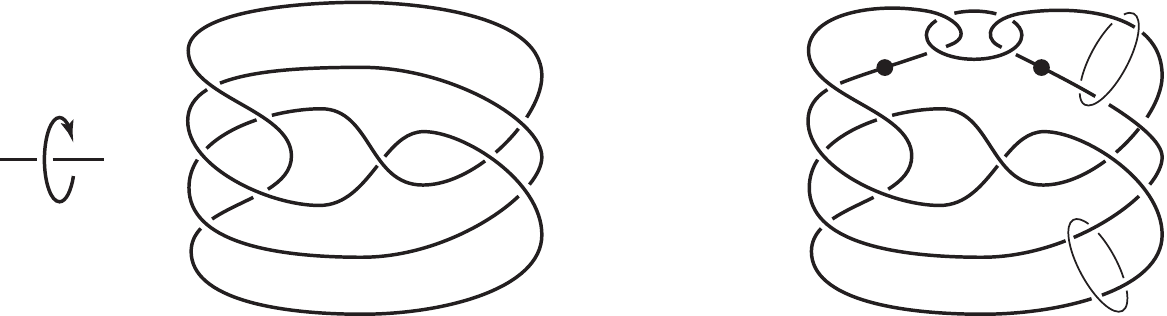}
\put(11,16){$\iota$}
\put(30,60){$0$}
\put(210,70){$0$}
\put(75,-10){$\partial W_1$}

\put(245,68){$c$}
\put(245,0){$\iota(c)$}
\put(208,-10){$W_1$}
\end{overpic}
\vspace{3mm}
\caption{An involution $\iota$ on $\partial W_1$.}
\label{fig: involution}
\end{figure}

\end{proposition}

\begin{proof}
Let $c$ be a curve on $\partial W_1$ as in Figure~\ref{fig: involution}.
Then the image $\iota(c)$ by the map $\iota$
is indicated in Figure~\ref{fig: involution}.

Attaching a 4-dimensional 2-handle $h$ on $W_1$ along $c$ (with a framing), we obtain a diffeomorphism $\partial W_1\times [0,1]\cup h \to \partial W_1\times [0,1]\cup_\iota h'$.
Here $h'$ is a 2-handle along $\iota(c)$.
If $\iota$ can homotopically extend inside $W_1$, then we can get a homotopy equivalence $W_1\cup h\to W_1\cup_\iota h'$.
In particular, $\pi_1(W_1\cup h)$ and $\pi_1( W_1\cup_\iota h')$ must be isomorphic.

The former gives an isomorphism $\pi_1(W_1\cup h)\cong \pi_1(W_1)$.
In fact, because $c$ gives a trivial relator in $\pi_1(W_1)$, attaching of $h$ does not change $\pi_1(W_1)$.
Due to the diffeomorphism in Figure~\ref{fig: 946ak}, $\pi_{1}(W_{1})$ is isomorphic to $\langle x,y|y^{-1}xy=x^2\rangle$.
This group is a solvable Baumslag-Solitar group \cite{BS} and it is well-known that it is non-abelian.
On the other hand, the latter $\pi_1(W_1\cup_\iota h')$ is isomorphic to ${\mathbb Z}$.
See the last picture in Figure~\ref{fig: proof1}.

Thus, $\pi_1(W_1\cup h)$ and $\pi_1(W_1\cup_\iota h')$ are not isomorphic.
This contradiction implies that $\iota$ cannot extend to a self-homotopy equivalence from $W_1$ to itself.
\end{proof}

\begin{remark}
Using a similar argument, there is an involution between the boundaries of exteriors of the pair of the disks $(D_{11},D_{22})$ (and $(D_{12},D_{21})$) which cannot extend to the inside as a self-homotopy equivalence.
\end{remark}

%

\begin{remark}
{\normalfont
Our non-isotopy property of $D_1$ and $D_2$ can also be reinterpreted by the Akbulut cork $C$ in Figure~\ref{fig: akbulut}. This is an idea by Akbulut \cite{A}. The diffeomorphisms which can be seen in Figure~\ref{fig: 946ak} give a diffeomorphism between the exterior of the ribbon disk $D_i$ ($i=1,2$) in $B^4$ and the exterior of a slice disk $D$ in $C$. The spanning disk of right dotted circle in the last picture of Figure~\ref{fig: 946ak} is $D$. This diffeomorphism relates the symmetries of the ribbon disks $D_1$ and $D_2$ of $\overline{9_{46}}$ and the cork twist of $C$. Let $\gamma$ and $\delta$ be the meridians of a dotted 1-handle and a 2-handle of the handle decomposition of $C$. See the last picture of Figure~\ref{fig: 946ak}. The two meridians are isotopic to the positions in the first picture in Figure~\ref{fig: 946ak} by the diffeomorphism. Furthermore, the meridians are moved to a symmetric position in $W_1$ and $W_2$ by an isotopy.  In Figure~\ref{fig: sym} we draw the positions of $\gamma$ and $\delta$ in $S^3\setminus \overline{9_{46}}$.
Thus, we can understand that the non-isotopiness of two disks is caused by the non-extendability property of the cork twist.
}
\begin{figure}[htpb]
\begin{overpic}
[width=.8\textwidth]
{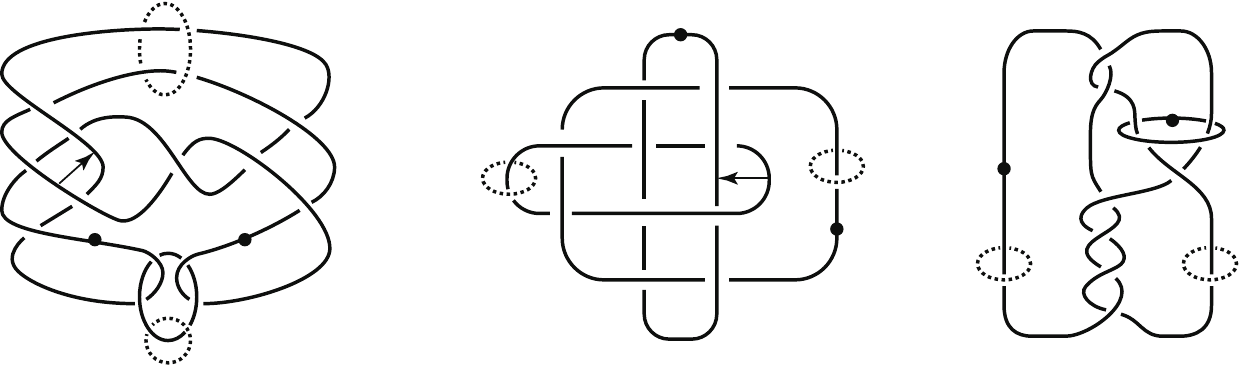}
\put(38, 30){$0$}
\put(44, 80){$\delta$}
\put(39, -5){$\gamma$}
\put(90, 40){$\cong$}
\put(108, 34){$\gamma$}
\put(123, 60){$0$}
\put(198, 50){$\delta$}
\put(210, 40){$\cong$}
\put(219, 20){$\delta$}
\put(283, 70){$0$}
\put(290, 20){$\gamma$}
\end{overpic}
\caption{The exterior of the ribbon disk $D_2$ of $\overline{9_{46}}$ in $B^4$ is diffeomorphic to the exterior of the slice disk in Akbulut cork.}
\label{fig: 946ak}
\end{figure}
\begin{figure}[htpb]
\begin{overpic}
[width=.35\textwidth]
{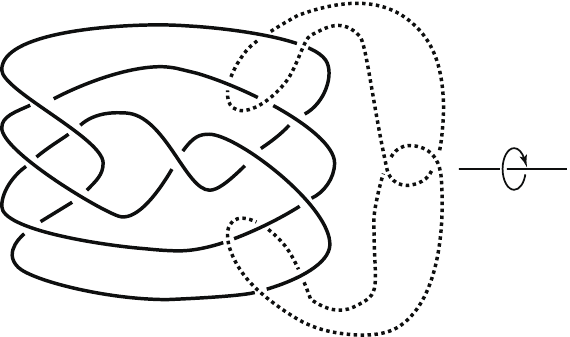}
\put(100, 60){$\delta$}
\put(100, 5){$\gamma$}
\end{overpic}
\caption{The symmetric positions of $\gamma$ and $\delta$ in $S^3\setminus \overline{9_{46}}$.}
\label{fig: sym}
\end{figure}

\end{remark}
\begin{remark}
{\normalfont
Both of Akbulut's slice disks $d_1,d_2$ are not isotopic to any Lagrangian disk even if taking the mirror image.
In fact, 
any ribbon disk in $B^4$ of both of $K_{Ak}$ and the mirror image $\overline{K_{Ak}}$ is not isotopic to a Lagrangian disk.
If $K_{Ak}$ has a Legendrian knot $L$ that fills a Lagrangian disk, then $tb(L)=-1$ due to \cite{Ch}.
By \cite{Ru},  $tb(L)$ has an upper bound $tb(L)\le \min\deg_a(F_{K_{Ak}})-1$.
Here $\min\deg_a(F_{K_{Ak}})$ is the minimal $a$-degree of the Kauffman polynomial $F_{K_{Ak}}(a,z)$.
By an easy calculation using the skein relation, $\min\deg_a(F_{K_{Ak}})=-1$. So  $tb(L)\le -2$, and there is no Lagrangian disk filling for $K$. Since $\min\deg_a(F_{\overline{K_{Ak}}})=-8$, $\overline{K_{Ak}}$ has no Lagrangian disk filling for the same reason.

}
\end{remark}

\end{document}